\documentclass{amsart}

\usepackage{amsmath}
\usepackage{amssymb}

\allowdisplaybreaks

\newtheorem{thm}{Theorem}[section]
\newtheorem{cor}[thm]{Corollary}
\newtheorem{lem}[thm]{Lemma}
\newtheorem{propo}[thm]{Proposition}

\theoremstyle{definition}

\theoremstyle{remark}

\title[Littlewood-Paley-Stein $g_k$-functions]{Discrete harmonic analysis \\associated with Jacobi expansions III: \\ the Littlewood-Paley-Stein $g_k$-functions and the Laplace type multipliers}

\author[A. Arenas]{Alberto Arenas}
\address{Departamento de Matem\'aticas y Computaci\'on,
Universidad de La Rioja, Complejo Cient\'{\i}fico-Tecnol\'ogico,
Calle Madre de Dios 53, 26006 Logro\~no, Spain}
\email{alberto.arenas@unirioja.es}

\author[\'O. Ciaurri]{\'Oscar Ciaurri}
\address{Departamento de Matem\'aticas y Computaci\'on,
Universidad de La Rioja, Complejo Cient\'{\i}fico-Tecnol\'ogico,
Calle MAdre de Dios 53, 26006 Logro\~no, Spain}
\email{oscar.ciaurri@unirioja.es}

\author[E. Labarga]{Edgar Labarga}
\address{Departamento de Matem\'aticas y Computaci\'on,
Universidad de La Rioja, Complejo Cient\'{\i}fico-Tecnol\'ogico,
Calle Madre de Dios 53, 26006 Logro\~no, Spain}
\email{edgar.labarga@unirioja.es}

\thanks{The first-named author was supported by a predoctoral research grant of the Government of Comunidad Aut\'{o}noma de La Rioja. The second-named author was supported by grant MTM2015-65888-C04-4-P MINECO/FEDER, UE, from Spanish Government. The third-named author was supported by a predoctoral research grant of the University of La Rioja.}

\begin{document}

\begin{abstract}
The research about harmonic analysis associated with Jacobi expansions carried out in \cite{ACL-JacI} and \cite{ACL-JacII} is continued in this paper. Given the operator $\mathcal{J}^{(\alpha,\beta)}=J^{(\alpha,\beta)}-I$, where $J^{(\alpha,\beta)}$ is the three-term recurrence relation for the normalized Jacobi polynomials and $I$ is the identity operator, we define the corresponding Littlewood-Paley-Stein $g_k^{(\alpha,\beta)}$-functions associated with it and we prove an equivalence of norms with weights for them. As a consequence, we deduce a result for Laplace type multipliers.
\end{abstract}

\maketitle

\section{Introduction}
We begin by setting some aspects of our context as in the previous papers \cite{ACL-JacI, ACL-JacII}. For $\alpha,\beta>-1$, we take the sequences  $\{a_{n}^{(\alpha,\beta)}\}_{n\in\mathbb{N}}$ and $\{b_{n}^{(\alpha,\beta)}\}_{n\in\mathbb{N}}$ given by
\[
a_n^{(\alpha,\beta)}=\frac{2}{2n+\alpha+\beta+2}
\sqrt{\frac{(n+1)(n+\alpha+1)(n+\beta+1)(n+\alpha+\beta+1)}{(2n+\alpha+\beta+1)(2n+\alpha+\beta+3)}},\quad n\geq 1,
\]
\[
a_{0}^{(\alpha,\beta)} = \frac{2}{\alpha+\beta+2}\sqrt{\frac{(\alpha+1)(\beta+1)}{(\alpha+\beta+3)}},
\]
\[
b_n^{(\alpha,\beta)}=\frac{\beta^2-\alpha^2}{(2n+\alpha+\beta)(2n+\alpha+\beta+2)},\quad n\geq 1,
\]
and
\[
b_{0}^{(\alpha,\beta)} = \frac{\beta-\alpha}{\alpha+\beta+2}.
\]
For each sequence $\{f(n)\}_{n\ge 0}$, we define the operator $\{J^{(\alpha,\beta)}f(n)\}_{n\ge 0}$ by the relation
\[
J^{(\alpha,\beta)}f(n)=a_{n-1}^{(\alpha,\beta)}f(n-1)+b_n^{(\alpha,\beta)}f(n)+ a_{n}^{(\alpha,\beta)}f(n+1), \qquad n\ge 1,
\]
and $J^{(\alpha,\beta)}f(0)=b_0^{(\alpha,\beta)}f(0)+ a_{0}^{(\alpha,\beta)}f(1)$.

Defining the Jacobi polynomials $\{P^{(\alpha,\beta)}_n(x)\}_{n\ge 0}$ through the Rodrigues' formula (see \cite[p.~67, eq.~(4.3.1)]{Szego})
\[
P_n^{(\alpha,\beta)}(x)=\frac{(-1)^n}{2^n \, n!}(1-x)^{-\alpha}(1+x)^{-\beta}\frac{d^n}{dx^n}\left\{(1-x)^{\alpha+n}(1+x)^{\beta+n}\right\},
\]
it is well known that they are orthogonal on the interval $[-1,1]$ with respect to the measure
\[
d\mu_{\alpha,\beta}(x)=(1-x)^\alpha(1+x)^{\beta}\,dx.
\]
Moreover, the sequence $\{p_n^{(\alpha,\beta)}(x)\}_{n\ge 0}$, given by $p_n^{(\alpha,\beta)}(x)=w_n^{(\alpha,\beta)}P_n^{(\alpha,\beta)}(x)$ where
\begin{equation*}
\begin{aligned}
w_n^{(\alpha,\beta)}& = \frac{1}{\|P_n^{(\alpha,\beta)}\|_{L^2((-1,1),d\mu_{\alpha,\beta})}} \\&= \sqrt{\frac{(2n+\alpha+\beta+1)\, n!\,\Gamma(n+\alpha+\beta+1)}{2^{\alpha+\beta+1}\Gamma(n+\alpha+1)\,\Gamma(n+\beta+1)}},\quad n\geq1,
\end{aligned}
\end{equation*}
and
\[
w_{0}^{(\alpha,\beta)} = \frac{1}{\|P_{0}^{(\alpha,\beta)}\|_{L^2((-1,1),d\mu_{\alpha,\beta})}} = \sqrt{\frac{\Gamma(\alpha+\beta+2)}{2^{\alpha+\beta+1}\Gamma(\alpha+1)\Gamma(\beta+1)}},
\]
is an orthonormal and complete system in $L^2((-1,1),d\mu_{\alpha,\beta})$, and it satisfies that
\[
J^{(\alpha,\beta)}p^{(\alpha,\beta)}_n(x)=xp_n^{(\alpha,\beta)}(x).
\]
Along this paper we will work with the operator
\[
\mathcal{J}^{(\alpha,\beta)}f(n)=(J^{(\alpha,\beta)}-I)f(n), \qquad x\in [-1,1],
\]
where $I$ denotes the identity operator, instead of $J^{(\alpha,\beta)}$ since the translated operator $-\mathcal{J}^{(\alpha,\beta)}$ is non-negative. In fact, the spectrum of $J^{(\alpha,\beta)}$ is the interval $[-1,1]$, so that the spectrum of $-\mathcal{J}^{(\alpha,\beta)}$ is $[0,2]$.

This paper continues in a natural way the study of harmonic analysis associated with $\mathcal{J}^{(\alpha,\beta)}$ of \cite{ACL-JacI} and \cite{ACL-JacII}. In \cite{ACL-JacI} we carried out  an exhaustive analysis of the heat semigroup for $\mathcal{J}^{(\alpha,\beta)}$ and in \cite{ACL-JacII} we investigated the Riesz transform. The main aim of this paper is to study another classical operator in harmonic analysis, the Littlewood-Paley-Stein $g_k$-function.

For an appropriate sequence $\{f(n)\}_{n\in \mathbb{N}}$ and $t>0$, the heat semigroup associated with $\mathcal{J}^{(\alpha,\beta)}$ is defined by the identity
\[
W_t^{(\alpha,\beta)}f(n)=\sum_{m=0}^\infty f(m)K^{(\alpha,\beta)}_t(m,n),
\]
where
\[
K_t^{(\alpha,\beta)}(m,n)=\int_{-1}^1 e^{-(1-x)t}p_m^{(\alpha,\beta)}(x)p_n^{(\alpha,\beta)}(x)\, d\mu_{\alpha,\beta}(x).
\]
Then, the Littlewood-Paley-Stein $g_k^{(\alpha,\beta)}$-functions in this context are given by
\begin{equation}
\label{eq:def-gk}
g^{(\alpha,\beta)}_k(f)(n)=\left(\int_{0}^{\infty}t^{2k-1}\left|\frac{\partial^k}{\partial t^k}W_t^{(\alpha,\beta)}f(n)\right|^2\, dt\right)^{1/2}, \qquad k\ge 1.
\end{equation}

The history of $g$-functions goes back to the seminal paper by J. E. Littlewood and R. E. A. C. Paley \cite{LP}, published in 1936, where they introduced the $g$-function for the trigonometric Fourier series. The extension to the Fourier transform on $\mathbb{R}^n$ was given by E. M. Stein in \cite{Stein-TAMS} more than twenty years later. He himself treated the question in a very abstract setting in \cite{Stein-Rojo}. In the last few years, there has been a deep research of these operators in different contexts and considering weights. For example, for the Hankel transform they were studied in \cite{BCN}, for Jacobi expansions in \cite{NSS}, for Laguerre expansions in \cite{NStS}, for Hermite expansions in \cite{ST}, and for Fourier-Bessel expansions in \cite{Ciau-Ron}.

Our work on discrete harmonic analysis related to Jacobi polynomials pretends to be a generalization of the work in \cite{Ciau-et-al} for the discrete Laplacian
\begin{equation}
\label{eq:dis-lap}
\Delta_d f(n)=f(n-1)-2f(n)+f(n+1)
\end{equation}
and in \cite{Bet-et-al} in the case of ultraspherical expansions, which corresponds with the case $\alpha=\beta=\lambda-1/2$ of $J^{(\alpha,\beta)}$. In both cases the corresponding $g_k$-functions were analysed (in \cite{Ciau-et-al} only for $k=1$).

To present our main result we need to introduce some notation. A weight on $\mathbb{N}$ will be a strictly positive sequence $w=\{w(n)\}_{n\ge 0}$. We consider the weighted $\ell^{p}$-spaces
\[
\ell^p(\mathbb{N},w)=\left\{f=\{f(n)\}_{n\ge 0}: \|f\|_{\ell^{p}(\mathbb{N},w)}:=\Bigg(\sum_{m=0}^{\infty}|f(m)|^p w(m)\Bigg)^{1/p}<\infty\right\},
\]
$1\le p<\infty$, and we simply write $\ell^p(\mathbb{N})$ when $w(n)=1$ for all $n\in \mathbb{N}$.

Furthermore, we say that a weight $w(n)$ belongs to the discrete Muckenhoupt $A_p(\mathbb{N})$ class, $1<p<\infty$, provided that
\[
\sup_{\begin{smallmatrix} 0\le n \le m \\ n,m\in \mathbb{N} \end{smallmatrix}} \frac{1}{(m-n+1)^p}\Bigg(\sum_{k=n}^mw(k)\Bigg)\Bigg(\sum_{k=n}^mw(k)^{-1/(p-1)}\Bigg)^{p-1} <\infty,
\]
holds.

The main result of this paper is the following one. 

\begin{thm}
\label{th:main}
Let $\alpha,\beta\ge -1/2$, $1<p<\infty$, and $w\in A_p(\mathbb{N})$. Then,
\begin{equation}
\label{eq:bound-gk}
C_1 \|f\|_{\ell^p(\mathbb{N},w)}\le \|g_k^{(\alpha,\beta)}(f)\|_{\ell^p(\mathbb{N},w)}\le C_2 \|f\|_{\ell^p(\mathbb{N},w)},\qquad f\in\ell^2(\mathbb{N})\cap\ell^{p}(\mathbb{N},w),
\end{equation}
where $C_1$ and $C_2$ are constants independent of $f$.
\end{thm}

To prove this theorem we will start by showing that the second inequality in \eqref{eq:bound-gk} implies the first one. After two appropriate reductions, the former will be deduced from the case $(\alpha,\beta)=(-1/2,-1/2)$ and $k=1$ that we will obtain from  discrete Calder\'on-Zygmund theory.

It is very common to define $g_k$-functions in terms of the Poisson semigroup instead of the heat semigroup. In our case the Poisson semigroup can be defined by subordination through the identity
\begin{equation}
\label{eq:poisson}
P^{(\alpha,\beta)}_tf(n)=\frac{1}{\sqrt{\pi}}\int_{0}^{\infty}\frac{e^{-u}}{\sqrt{u}}W_{t^2/(4u)}^{(\alpha,\beta)}f(n)\, du,\quad t\ge 0,
\end{equation}
and then we have the $\mathfrak{g}_k^{(\alpha,\beta)}$-function
\begin{equation*}
\mathfrak{g}_k^{(\alpha,\beta)}(f)(n)=\left(\int_{0}^{\infty}t^{2k-1}\left|\frac{\partial^k }{\partial t^k}P^{(\alpha,\beta)}_tf(n)\right|^2\, dt\right)^{1/2}, \qquad k\ge 1.
\end{equation*}
The following result will be a consequence of Theorem \ref{th:main}.
\begin{cor}
\label{cor:poisson}
Let $\alpha,\beta\ge -1/2$, $1<p<\infty$, and $w\in A_p(\mathbb{N})$. Then,
\begin{equation*}
C_1 \|f\|_{\ell^p(\mathbb{N},w)}\le \|\mathfrak{g}_k^{(\alpha,\beta)}(f)\|_{\ell^p(\mathbb{N},w)}\le C_2 \|f\|_{\ell^p(\mathbb{N},w)},\qquad f\in\ell^2(\mathbb{N})\cap\ell^{p}(\mathbb{N},w),
\end{equation*}
where $C_1$ and $C_2$ are constants independent of $f$.
\end{cor}

We will prove this corollary by controlling the $\mathfrak{g}_k^{(\alpha,\beta)}$-function by a finite sum of $g_k^{(\alpha,\beta)}$-functions (see Lemma \ref{lem:gkPoi-gkHeat}). This fact will follow from the subordination identity \eqref{eq:poisson}.

As an application of Theorem \ref{th:main}, we will prove the boundedness of some multipliers of Laplace type for the discrete Fourier-Jacobi series. As it is well known, for each function $f\in L^2([-1,1],d\mu_{\alpha,\beta})$ its Fourier-Jacobi coefficients are given by
\[
c_m^{(\alpha,\beta)}(f)=\int_{-1}^{1}f(x)p_m^{(\alpha,\beta)}(x)\,d\mu_{\alpha,\beta}(x)
\]
and
\[
f(x)=\sum_{m=0}^{\infty}c_m^{(\alpha,\beta)}(f)p_m^{(\alpha,\beta)}(x),
\]
where the equality holds in $ L^2([-1,1],d\mu_{\alpha,\beta})$. Moreover, $\{c_m^{(\alpha,\beta)}(f)\}_{m\ge 0}$ is a sequence in $\ell^{2}(\mathbb{N})$. Conversely, for each sequence $f\in \ell^2(\mathbb{N})$, the function
\begin{equation}
\label{eq:Jac-transform}
F_{\alpha,\beta}(x)=\sum_{m=0}^{\infty}f(m)p_m^{(\alpha,\beta)}(x)
\end{equation}
belongs to $L^2([-1,1],d\mu_{\alpha,\beta})$ and the Parseval's identity
\begin{equation}
\label{ec:Plancherel}
\|f\|_{\ell^2(\mathbb{N})}=\|F_{\alpha,\beta}\|_{L^2([-1,1],d\mu_{\alpha,\beta})}
\end{equation}
holds. Moreover, $c_m^{(\alpha,\beta)}(F_{\alpha,\beta})=f(m)$.

Note that an obvious consequence of \eqref{ec:Plancherel} is the useful relation
\begin{equation}
\label{eq:parseval}
\sum_{m=0}^{\infty} f(m)g(m)=\int_{-1}^{1}F_{\alpha,\beta}(x)G_{\alpha,\beta}(x)\, d\mu_{\alpha,\beta}(x),\qquad f,g\in \ell^2(\mathbb{N}),
\end{equation}
where $F_{\alpha,\beta}$ is given by \eqref{eq:Jac-transform} and $G_{\alpha,\beta}$ is defined in a similar way.

Given a bounded function $M$ defined on $[0,2]$, the multiplier associated with $M$ is the operator, initially defined on $\ell^{2}(\mathbb{N})$, by the identity
\[
T_Mf(n)=c_n^{(\alpha,\beta)}(M(1-\cdot)F_{\alpha,\beta}).
\]
We say that $T_M$ is a Laplace type multiplier when
\[
M(x)=x \int_{0}^{\infty}e^{-xt}a(t)\, dt,
\]
with $a$ being a bounded function. From a spectral point of view, $T_M=M(\mathcal{J}^{(\alpha,\beta)})$.

The Laplace type multipliers were introduced by Stein in \cite[Ch. 2]{Stein-Rojo}. There, it is observed that they verify $|x^k M^{(k)}(x)|\le C_k$ for $k=0,1,\dots$, and then form a subclass of Marcinkiewicz multipliers. For the operators $T_M$ we have the following result.
\begin{thm}
\label{th:laplace-multi}
  Let $\alpha,\beta\ge -1/2$, $1<p<\infty$, and $w\in A_p(\mathbb{N})$. Then,
\[
\|T_Mf\|_{\ell^p(\mathbb{N},w)}\le C\|f\|_{\ell^p(\mathbb{N},w)},\qquad  f\in \ell^2(\mathbb{N})\cap\ell^{p}(\mathbb{N},w),
\end{equation*}
where $C$ is a constant independent of $f$.
\end{thm}

From the identity,
\[
x^{i\gamma}=\frac{x}{\Gamma(1-i\gamma)}\int_{0}^{\infty}e^{-xt}t^{-i\gamma},\qquad \gamma \in \mathbb{R},
\]
we deduce the following corollary.
\begin{cor}
  Let $\alpha,\beta\ge -1/2$, $1<p<\infty$, and $w\in A_p(\mathbb{N})$. Then, 
\[
\|(\mathcal{J}^{(\alpha,\beta)})^{i\gamma}f\|_{\ell^p(\mathbb{N},w)}\le C\|f\|_{\ell^p(\mathbb{N},w)}, \qquad  f\in \ell^2(\mathbb{N})\cap\ell^{p}(\mathbb{N},w),
\end{equation*}
where $C$ is a constant independent of $f$.
\end{cor}

The rest of the paper is organized as follows. Section \ref{sec:Proof-Main} contains the proof of Theorem \ref{th:main} which relies on a transplantation theorem and the Calder\'on-Zygmund theory. The proofs of two propositions that are necessary to apply the Calder\'on-Zygmund theory are provided in Section \ref{sec:size-smooth}. Section \ref{sec:poisson} and Section \ref{sec:Proof-Laplace} contain the proofs of Corollary \ref{cor:poisson} and Theorem \ref{th:laplace-multi}, respectively.

\section{Proof of Theorem \ref{th:main}}
\label{sec:Proof-Main}
We consider the Banach space $\mathbb{B}_k=L^2\left((0,\infty), t^{2k-1}\, dt\right)$, with $k\ge 1$, and the operator
\[
G^{(\alpha,\beta)}_{t,k}f(n)=\sum_{m= 0}^{\infty} f(m)G^{(\alpha,\beta)}_{t,k}(m,n),
\]
with
\begin{align*}
G^{(\alpha,\beta)}_{t,k}(m,n)&=\frac{\partial^k}{\partial t^k} K^{(\alpha,\beta)}_t(m,n)\\&=(-1)^k\int_{-1}^{1}(1-x)^k e^{-t(1-x)}p_m^{(\alpha,\beta)}(x)p_n^{(\alpha,\beta)}(x)\, d\mu_{\alpha,\beta}(x).
\end{align*}
Then, it is clear that
\[
g^{(\alpha,\beta)}_k(f)(n)=\left\|G^{(\alpha,\beta)}_{t,k}f(n)\right \|_{\mathbb{B}_k}.
\]

A first tool to prove Theorem \ref{th:main} is the following result about the $\ell^2$-boundedness of the $g_k^{(\alpha,\beta)}$-functions.

\begin{lem}
\label{lem:gk-L2}
  Let $\alpha,\beta\ge -1/2$ and $k\ge 1$. Then, 
\begin{equation}
\label{eq:gk-L2}
\|g_k^{(\alpha,\beta)}(f)\|^2_{\ell^2(\mathbb{N})}=\frac{\Gamma(2k)}{2^{2k}}\|f\|^2_{\ell^2(\mathbb{N})}.
\end{equation}
\end{lem}
\begin{proof}
For a sequence $f\in \ell^{2}(\mathbb{N})$, it is satisfied that
\begin{align*}
G_{t,k}^{(\alpha,\beta)} f(n)&=(-1)^k\int_{-1}^{1}(1-x)^k e^{-t(1-x)}F_{\alpha,\beta}(x)p_n^{(\alpha,\beta)}(x)\, d\mu_{\alpha,\beta}(x)
\\&=(-1)^k c_n^{(\alpha,\beta)}((1-\cdot)^k e^{-t(1-\cdot)}F_{\alpha,\beta}).
\end{align*}
Then, by using \eqref{ec:Plancherel}, we have
\begin{align*}
\|g_k^{(\alpha,\beta)}(f)\|^2_{\ell^2(\mathbb{N})}&=\sum_{n=0}^{\infty}\int_{0}^{\infty}t^{2k-1} (c_n^{(\alpha,\beta)}((1-\cdot)^k e^{-t(1-\cdot)}F_{\alpha,\beta}))^2\, dt
\\& =\int_{0}^{\infty}t^{2k-1}\sum_{n=0}^{\infty}(c_n^{(\alpha,\beta)}((1-\cdot)^k e^{-t(1-\cdot)}F_{\alpha,\beta}))^2\, dt
\\&=\int_{0}^{\infty}t^{2k-1}\int_{-1}^{1}(1-x)^{2k} e^{-2t(1-x)}(F_{\alpha,\beta}(x))^2\, d\mu_{\alpha,\beta}(x)\, dt\\&=\int_{-1}^{1}(1-x)^{2k}(F_{\alpha,\beta}(x))^2\int_{0}^{\infty}t^{2k-1} e^{-2t(1-x)}\, dt\, d\mu_{\alpha,\beta}(x)\\&=\frac{\Gamma(2k)}{2^{2k}}\int_{-1}^{1}(F_{\alpha,\beta}(x))^2d\mu_{\alpha,\beta}(x)=
\frac{\Gamma(2k)}{2^{2k}}\|f\|^2_{\ell^2(\mathbb{N})}
\end{align*}
and the proof is completed.
\end{proof}

Now, let us see that
\begin{equation}
\label{eq:gk-direct}
\|g_k^{(\alpha,\beta)}(f)\|_{\ell^p(\mathbb{N},w)}\le C \|f\|_{\ell^p(\mathbb{N},w)}
\end{equation}
implies the reverse inequality
\begin{equation}
\label{eq:gk-reverse}
\|f\|_{\ell^p(\mathbb{N},w)}\le C \|g_k^{(\alpha,\beta)}(f)\|_{\ell^p(\mathbb{N},w)}.
\end{equation}
Polarising the identity \eqref{eq:gk-L2}, we have
\[
\sum_{n=0}^{\infty}f(n)h(n)=\frac{2^{2k}}{\Gamma(2k)}\sum_{n=0}^{\infty} \int_{0}^{\infty}t^{2k-1}\left(\frac{\partial^k}{\partial t^k}W_t^{(\alpha,\beta)}f(n)\right)\left(\frac{\partial^k}{\partial t^k}W_t^{(\alpha,\beta)}h(n)\right)\, dt
\]
and, obviously,
\[
\left|\sum_{n=0}^{\infty}f(n)h(n)\right|\le C \sum_{n=0}^{\infty}g_k^{(\alpha,\beta)}(f)(n)g_k^{(\alpha,\beta)}(h)(n).
\]
Taking $h(n)=w^{1/p}(n)f_1(n)$, we have
\begin{align*}
\left|\sum_{n=0}^{\infty}f(n)w^{1/p}(n)f_1(n)\right|&\le C
\sum_{n=0}^{\infty}g_k^{(\alpha,\beta)}(f)(n)g_k^{(\alpha,\beta)}(w^{1/p}f_1)(n)
\\&=C\sum_{n=0}^{\infty}g_k^{(\alpha,\beta)}(f)(n)w^{1/p}(n)w^{-1/p}(n)g_k^{(\alpha,\beta)}(w^{1/p}f_1)(n)\\&\le
C\|g_k^{(\alpha,\beta)}(f)\|_{\ell^p(\mathbb{N},w)}\|g_k^{(\alpha,\beta)}(w^{1/p}f_1)\|_{\ell^{p'}(\mathbb{N},w')},
\end{align*}
where $w'=w^{-1/(p-1)}$ and $p'$ is the conjugate exponent of $p$; i.e., $1/p+1/p'=1$. Note that $w\in A_p(\mathbb{N})$ implies $w'\in A_{p'}(\mathbb{N})$ and, by \eqref{eq:gk-direct},
\[
\|g_k^{(\alpha,\beta)}(w^{1/p}f_1)\|_{\ell^{p'}(\mathbb{N},w')}\le C \|w^{1/p}f_1\|_{\ell^{p'}(\mathbb{N},w')}=\|f_1\|_{\ell^{p'}(\mathbb{N})}.
\]
So, we obtain that
\[
\left|\sum_{n=0}^{\infty}f(n)w^{1/p}(n)f_1(n)\right|\le C \|g^{(\alpha,\beta)}_k(f)\|_{\ell^p(\mathbb{N},w)}\|f_1\|_{\ell^{p'}(\mathbb{N})}
\]
and taking the supremum over all $f_1\in \ell^{p'}(\mathbb{N})$ such that $\|f_1\|_{\ell^{p'}(\mathbb{N})}\le 1$, we conclude the inequality \eqref{eq:gk-reverse}.

In this way, we have reduced the proof of Theorem \ref{th:main} to prove \eqref{eq:gk-direct}. Now, we proceed with two new reductions. First, we are going to use a proper transplantation operator to deduce \eqref{eq:gk-direct} from the case $(\alpha,\beta)=(-1/2,-1/2)$ for $k\ge 1$. Finally, we will see how to obtain \eqref{eq:gk-direct} for $g_k^{(-1/2,-1/2)}$ with $k> 1$ from the case $k=1$. These reductions in the proof are inspired by the work in \cite{G-T-et-al}.

For $f\in \ell^2(\mathbb{N})$ we define the transplantation operator
\[
T_{\alpha,\beta}^{\gamma,\delta}f(n)=\sum_{m=0}^{\infty} f(m)K_{\alpha,\beta}^{\gamma,\delta}(n,m)
\]
where
\[
K_{\alpha,\beta}^{\gamma,\delta}(n,m)=\int_{-1}^{1}p_n^{(\gamma,\delta)}(x)p_m^{(\alpha,\beta)}(x)\, \mu_{\gamma/2+\alpha/2,\delta/2+\beta/2}(x).
\]
This operator was analysed in \cite{ACL-Trans}, where an extension of a classical result from R.~Askey \cite{Askey-Jacobi} was given. In fact, it was proved that
\[
\|T_{\alpha,\beta}^{\gamma,\delta}f\|_{\ell^p(\mathbb{N}, w)}\le C \|f\|_{\ell^p(\mathbb{N}, w)}, \qquad 1<p<\infty,
\]
with weights $w\in A_p(\mathbb{N})$, and the analogous weak inequality from $\ell^1(\mathbb{N},w)$ into $\ell^{1,\infty}(\mathbb{N}, w)$ for weights in the $A_1(\mathbb{N})$ class. By a result due to Krivine (see \cite[Theorem 1.f.14]{L-Z}), it is possible to give, in an obvious way, a vector-valued extension of the transplantation operator to the space $\mathbb{B}_k$, denoted by $\overline{T}_{\alpha,\beta}^{\gamma,\delta}$, satisfying
\[
\|\overline{T}_{\alpha,\beta}^{\gamma,\delta}f\|_{\ell_{\mathbb{B}_k}^p(\mathbb{N}, w)}\le C \|f\|_{\ell_{\mathbb{B}_k}^p(\mathbb{N}, w)}, \qquad 1<p<\infty,
\]
with weights in $A_p(\mathbb{N})$.

In this way, we have
\begin{equation}
\label{eq:compo-gk}
G_{t,k}^{(\alpha,\beta)}f=\overline{T}^{\alpha,\beta}_{-1/2,-1/2}G_{t,k}^{(-1/2,-1/2)}{T}^{-1/2,-1/2}_{\alpha,\beta}f.
\end{equation}
Indeed, we have
\begin{align*}
G_{t,k}^{(-1/2,-1/2)}{T}^{-1/2,-1/2}_{\alpha,\beta}f(n)
&=\frac{\partial^k}{\partial t^k}W_t^{(-1/2,-1/2)}{T}^{-1/2,-1/2}_{\alpha,\beta}f(n)
\\&=
\sum_{m=0}^{\infty}f(m)\sum_{j=0}^{\infty}G_{t,k}^{(-1/2,-1/2)}(j,n)K^{-1/2,-1/2}_{\alpha,\beta}(j,m)
\end{align*}
and, by using \eqref{eq:parseval} and the identities
\[
G_{t,k}^{(-1/2,-1/2)}(j,n)=(-1)^k c_j^{(-1/2,-1/2)}((1-\cdot)^k e^{-t(1-\cdot)}p_n^{(-1/2,-1/2)})
\]
and
\[
K^{-1/2,-1/2}_{\alpha,\beta}(j,m)=c_j^{(-1/2,-1/2)}((1-\cdot)^{\alpha/2+1/4}(1+\cdot)^{\beta/2+1/4}p_m^{(\alpha,\beta)}),
\]
we deduce that
\begin{multline*}
\sum_{j=0}^{\infty}G_{t,k}^{(-1/2,-1/2)}(j,n)K^{-1/2,-1/2}_{\alpha,\beta}(j,m)\\=(-1)^k
\int_{-1}^{1}(1-x)^k e^{-t(1-x)}p_n^{(-1/2,-1/2)}(x)p_m^{(\alpha,\beta)}(x)\, d\mu_{\alpha/2-1/4,\beta/2-1/4}(x).
\end{multline*}
Applying a similar argument to the other composition the proof of \eqref{eq:compo-gk} follows.

Now, let us see that it is enough to analyse the $g_1^{(-1/2,-1/2)}$-function. In fact, using induction we can deduce the boundedness of the $g_k^{(-1/2,-1/2)}$-functions for $k>1$. Let us suppose that the operator $G_{t,k}^{(-1/2,-1/2)}$ is bounded from $\ell^p(\mathbb{N},w)$ into $\ell^p_{\mathbb{B}_k}(\mathbb{N},w)$. Taking $k=1$ and applying again Krivine's theorem, we deduce that the operator $\overline{G}^{(-1/2,-1/2)}_{t,1}: \ell^p_{\mathbb{B}_k}(\mathbb{N},w)\longrightarrow \ell^p_{\mathbb{B}_k\times \mathbb{B}_1}(\mathbb{N},w)$, given by
\[
\{f_s(n)\}_{s\ge 0}\longmapsto \{G_{t,1}^{(-1/2,-1/2)}f_s\}_{t,s\ge 0},
\]
is bounded. Moreover, $\overline{G}_{t,1}^{(-1/2,-1/2)}\circ G_{s,k}^{(-1/2,-1/2)}$ is a bounded operator from $\ell^p(\mathbb{N},w)$ into $\ell^p_{\mathbb{B}_k\times \mathbb{B}_1}(\mathbb{N},w)$. Now, using the identity
\[
\frac{\partial}{\partial t}\left(W_t^{(-1/2,-1/2)}\left(\frac{\partial^k}{\partial s^k}W_s^{(-1/2,-1/2)}f\right)\right)=\left.\frac{\partial^{k+1}}{\partial u^{k+1}}W_u^{(-1/2,-1/2)}f\right|_{u=s+t},
\]
we have
\begin{multline*}
\left\|\overline{G}_{t,1}^{(-1/2,-1/2)}\circ G_{s,k}^{(-1/2,-1/2)}f\right\|_{\mathbb{B}_k\times \mathbb{B}_1}^2\\
\begin{aligned}
&=\int_{0}^{\infty}\int_{0}^{\infty}ts^{2k-1}\left|\left.\frac{\partial^{k+1}}{\partial u^{k+1}}W_u^{(-1/2,-1/2)}f\right|_{u=s+t}\right|^2\,ds\, dt\\&=
\int_{0}^{\infty}\int_{t}^{\infty}t(r-t)^{2k-1}\left|\left.\frac{\partial^{k+1}}{\partial u^{k+1}}W_u^{(-1/2,-1/2)}f\right|_{u=r}\right|^2\,dr\, dt
\\&=\int_{0}^{\infty}\left|\frac{\partial^{k+1}}{\partial r^{k+1}}W_r^{(-1/2,-1/2)}f\right|^2\int_{0}^{r}t(r-t)^{2k-1}\,dt\, dr
\\&=\frac{1}{(2k+1)(2k)}\int_{0}^{\infty}r^{2k+1}\left|\frac{\partial^{k+1}}{\partial r^{k+1}}W_r^{(-1/2,-1/2)}f\right|^2\, dr
\\&=\frac{g_{k+1}^{(-1/2,-1/2)}(f)}{(2k+1)(2k)}.
\end{aligned}
\end{multline*}

Finally, to complete the proof of Theorem \ref{th:main} we have to prove \eqref{eq:gk-direct} for $(\alpha,\beta)=(-1/2,-1/2)$ and $k=1$. This fact will be a consequence of the following propositions.
\begin{propo}
\label{propo:g-size}
  Let $n,m\in \mathbb{N}$ with $n\not= m$. Then,
  \begin{equation}
  \label{eq:g-size}
  \|G_{t,1}^{(-1/2,-1/2)}(m,n)\|_{\mathbb{B}_1}\le C |n-m|^{-1}.
  \end{equation}
  Moreover
  \begin{equation}
  \label{eq:g-size-m}
  \|G_{t,1}^{(-1/2,-1/2)}(n,n)\|_{\mathbb{B}_1}\le C.
  \end{equation}
\end{propo}

\begin{propo}
\label{propo:g-smooth}
  Let $n,m\in \mathbb{N}$ with $n\not= m$. Then, 
  \begin{equation*}
  \|G_{t,1}^{(-1/2,-1/2)}(m+1,n)-G_{t,1}^{(-1/2,-1/2)}(m,n)\|_{\mathbb{B}_1}\le C |n-m|^{-2}
  \end{equation*}
and
 \begin{equation*}
  \|G_{t,1}^{(-1/2,-1/2)}(m,n+1)-G_{t,1}^{(-1/2,-1/2)}(m,n)\|_{\mathbb{B}_1}\le C |n-m|^{-2}.
  \end{equation*}
\end{propo}

The proof of these propositions is the most delicate part of the proof of Theorem~\ref{th:main}, so it is postponed to the next section.

Now, using the decomposition
\begin{multline*}
|g_1^{(-1/2,-1/2)}f(n)|\\
\begin{aligned}
&\le \left\|\sum_{\begin{smallmatrix}
                       m=0 \\
                       m\not=n
                     \end{smallmatrix}}^{\infty}f(m)G_{t,1}^{(-1/2,-1/2)}(m,n)\right\|_{\mathbb{B}_1}
                     +\left\|f(n)G_{t,1}^{(-1/2,-1/2)}(n,n)\right\|_{\mathbb{B}_1}
                     \\&:=T_1f(n)+T_2f(n),
\end{aligned}
\end{multline*}
we can apply \eqref{eq:g-size} of Proposition \ref{propo:g-size}, Proposition \ref{propo:g-smooth}, and Lemma \ref{lem:gk-L2} to deduce from the Calder\'on-Zygmund theory the inequality
\[
\|T_1f\|_{\ell^p(\mathbb{N},w)}\le C \|f\|_{\ell^p(\mathbb{N},w)},
\]
and \eqref{eq:g-size-m} to obtain that
\[
\|T_2f\|_{\ell^p(\mathbb{N},w)}\le C \|f\|_{\ell^p(\mathbb{N},w)},
\]
finishing the proof of Theorem \ref{th:main}.

\section{Proof of Proposition \ref{propo:g-size} and Proposition \ref{propo:g-smooth}}
\label{sec:size-smooth}
Denoting by $T_n$ the Chebyshev polynomials, we have
\[
p_n^{(-1/2,-1/2)}(x)=\sqrt{\frac{2}{\pi}}T_n(x)=\sqrt{\frac{2}{\pi}}\cos(n\theta),
\]
for $n\not=0$ and where $x=\cos\theta$, and $p_0^{(-1/2,-1/2)}(x)=\sqrt{\frac{1}{\pi}}T_0(x)=\sqrt{\frac{1}{\pi}}$. Then, the identity (see \cite[p. 456]{PBM1})
\[
\frac{1}{\pi}\int_0^\pi e^{z\cos \theta}\cos(m\theta)\, d\theta=I_m(z), \qquad |\arg(z)|<\pi,
\]
where $I_m$ denotes the Bessel function of imaginary argument of order $m$,
implies
\begin{equation}
\label{eq:kernel-I-1}
W_t^{(-1/2,-1/2)}(m,n)=e^{-t}(I_{m+n}(t)+I_{n-m}(t)), \qquad n,m\not=0,
\end{equation}
\begin{equation}
\label{eq:kernel-I-2}
W_t^{(-1/2,-1/2)}(m,0)=e^{-t}I_m(t), \qquad\text{ and }\qquad  W_t^{(-1/2,-1/2)}(0,n)=e^{-t}I_n(t).
\end{equation}
To simplify notation, we set $K_t(n)=e^{-t}I_n(t)$.

We note that the proofs of Propositions \ref{propo:g-size} and \ref{propo:g-smooth} are similar to the one given in \cite[Proposition 4]{Ciau-et-al} but we have included them  for a self-contained exposition of the paper and to fix some details.

\begin{proof}[Proof of Proposition \ref{propo:g-size}]
The identity \cite[eq. (10.29.1)]{NIST}
\[
2I_m(t)=I_{m+1}(t)+I_{m-1}(t)
\]
yields
\begin{equation}
\label{eq:partial-K}
\frac{\partial K_t(n)}{\partial t}=\frac{1}{2}(K_t(n+1)-2K_t(n)+K_t(n-1)), \qquad n\ge 1,
\end{equation}
and
\begin{equation}
\label{eq:partial-K0}
\frac{\partial K_t(0)}{\partial t}=K_t(1)-K_t(0).
\end{equation}

The next identity is known as Schl\"afli's integral
representation of Poisson type for modified Bessel functions (see \cite[eq. (5.10.22)]{Lebedev}):
\begin{equation}
\label{eq:Schlafli}
I_{\nu}(z) = \frac{z^{\nu}}{\sqrt{\pi}\,2^{\nu}\Gamma(\nu+1/2)}
\int_{-1}^1e^{-zs} (1-s^2)^{\nu-1/2}\,ds, \quad
|\arg z|<\pi, \quad \nu>-\frac12.
\end{equation}
Integrating by parts once and twice in~\eqref{eq:Schlafli}, we have, respectively, the identities
\begin{equation}
\label{eq:SchlafliII}
I_{\nu}(z) = -\frac{z^{\nu-1}}{\sqrt{\pi}\,2^{\nu-1}\Gamma(\nu-1/2)}
\int_{-1}^1e^{-zs} s(1-s^2)^{\nu-3/2}\,ds, \quad \nu>\frac12,
\end{equation}
and
\begin{equation}
\label{eq:SchlafliIII}
I_{\nu}(z)=\frac{z^{\nu-2}}{\sqrt{\pi}\,2^{\nu-2}\Gamma(\nu-3/2)}
\int_{-1}^1e^{-zs} \frac{1+zs}{z}s(1-s^2)^{\nu-5/2}\,ds, \quad \nu>\frac32.
\end{equation}
Then from \eqref{eq:partial-K}, using \eqref{eq:Schlafli}, \eqref{eq:SchlafliII}, and \eqref{eq:SchlafliIII} with $\nu=n-1$, $\nu=n$, and $\nu=n+1$, respectively, we deduce that for $n\ge 1$
\begin{equation*}
\frac{\partial K_t(n)}{\partial t}=\frac{1}{2}\left(I_{1,t}(n)+I_{2,t}(n)\right),
\end{equation*}
where
\[
I_{1,t}(n)=\frac{t^{n-2}}{\sqrt{\pi}2^{n-1}\Gamma(n-1/2)}\int_{-1}^{1}e^{-t(1+s)}s(1-s^2)^{n-3/2}\, ds
\]
and
\[
I_{2,t}(n)= \frac{t^{n-1}}{\sqrt{\pi}2^{n-1}\Gamma(n-1/2)}
\int_{-1}^{1}e^{-t(1+s)}(1+s)^2(1-s^2)^{n-3/2}\, ds.
\]
Now, for $n\ge 2$,
\begin{multline*}
\pi  (\Gamma(n-1/2))^2\|I_{1,t}(n)\|_{\mathbb{B}_1}^2\\
\begin{aligned}
&=\frac{1}{2^{2n-2}}\int_{0}^{\infty}t^{2n-3}\int_{-1}^{1}e^{-t(1+s)}s(1-s^2)^{n-3/2}\,ds \int_{-1}^{1}e^{-t(1+r)}r(1-r^2)^{n-3/2}\, dr\\&=\frac{1}{ 2^{2n-2} }\int_{-1}^{1}\int_{-1}^{1}sr(1-s^2)^{n-3/2}(1-r^2)^{n-3/2}\int_{0}^{\infty}t^{2n-3}e^{-t(2+s+r)}\, dt\, ds\, dr
\\&=\frac{\Gamma(2n-2)}{ 2^{2n-2}}\int_{-1}^{1}\int_{-1}^{1}\frac{sr(1-s^2)^{n-3/2}(1-r^2)^{n-3/2}}{(2-s-r)^{2n-2}}\, ds \, dr
\\&=\Gamma(2n-2)\int_{0}^{1}\int_{0}^{1}\frac{(2u-1)(2v-1)(u(1-u))^{n-3/2}(v(1-v))^{n-3/2}}{(u+v)^{2n-2}}\, du \, dv,
\end{aligned}
\end{multline*}
where in the last step, we have applied the change of variables $s=2u-1$ and $r=2v-1$, and
\begin{align*}
\|I_{1,t}(n)\|_{\mathbb{B}_1}^2&\le C\frac{\Gamma(2n-2)}{(\Gamma(n-1/2))^2}\int_{0}^{1}(v(1-v))^{n-3/2}\int_{0}^{1}\frac{u^{n-3/2}}{(u+v)^{2n-2}}\, du \, dv\\&
=C\frac{\Gamma(2n-2)}{(\Gamma(n-1/2))^2}\int_{0}^{1}(1-v)^{n-3/2}\int_{0}^{1/v}\frac{z^{n-3/2}}{(1+z)^{2n-2}}\, dz \, dv\\&\le
C\frac{\Gamma(2n-2)}{(\Gamma(n-1/2))^2}\left(\int_{0}^{1}(1-v)^{n-3/2}\, dv\right)\left(\int_{0}^{\infty}\frac{z^{n-3/2}}{(1+z)^{2n-2}}\, dz \right)\\&=\frac{C}{(n-1/2)^2}.
\end{align*}
In a similar way and again for $n\ge 2$, we obtain that
\begin{equation*}
\|I_{2,t}(n)\|_{\mathbb{B}_1}^2=
\frac{16\Gamma(2n)}{\pi  (\Gamma(n-1/2))^2}\int_{0}^{1}\int_{0}^{1}\frac{(uv)^{n+1/2}((1-u)(1-v))^{n-3/2}}{(u+v)^{2n}}\, du \, dv
\end{equation*}
and
\begin{align*}
\|I_{2,t}(n)\|_{\mathbb{B}_1}^2&\le C\frac{\Gamma(2n)}{(\Gamma(n-1/2))^2}\left(\int_{0}^{1}v^2(1-v)^{n-3/2}\, dv\right)\left(\int_{0}^{\infty}\frac{z^{n+1/2}}{(1+z)^{2n}}\, dz \right)\\&=\frac{C}{(n+1/2)^2}.
\end{align*}
Hence,
\begin{equation}
\label{eq:general}
\left\|\frac{\partial K_t(n) }{\partial t}\right\|_{\mathbb{B}_1}\le \frac{C}{n}, \qquad \text{for $n\ge 2$.}
\end{equation}

Now, we prove that
\begin{equation}
\label{eq:particular}
\left\|\frac{\partial K_t(0) }{\partial t}\right\|_{\mathbb{B}_1}+\left\|\frac{\partial K_t(1) }{\partial t}\right\|_{\mathbb{B}_1}\le C.
\end{equation}
By Minkowski's integral inequality, it is clear that
\begin{align*}
\left\|\frac{\partial K_t(0) }{\partial t}\right\|_{\mathbb{B}_1}&=\frac{1}{\pi}\left\|\frac{\partial }{\partial t}\int_{0}^{\pi}e^{-t(1-\cos \theta)}\, d\theta\right\|_{\mathbb{B}_1}\\&=\frac{1}{\pi}\left\|\int_{0}^{\pi}e^{-t(1-\cos \theta)}(1-\cos \theta)\, d\theta\right\|_{\mathbb{B}_1}\\&\le
\frac{1}{\pi}\int_{0}^{\pi}(1-\cos \theta)\left\|e^{-t(1-\cos \theta)}\right\|_{\mathbb{B}_1}\le C.
\end{align*}
Similarly, we obtain that $\left\|\frac{\partial K_t(1) }{\partial t}\right\|_{\mathbb{B}_1}\le C$ and the proof of \eqref{eq:particular} is finished.

Finally, using \eqref{eq:kernel-I-1}, \eqref{eq:kernel-I-2}, \eqref{eq:general}, \eqref{eq:particular}, and the identity
\begin{equation}
\label{eq:par-I}
I_{-n}(t)=I_n(t),
\end{equation}
we conclude the proof of the proposition.
\end{proof}

\begin{proof}[Proof of the Proposition \ref{propo:g-smooth}]
By using \eqref{eq:kernel-I-1}, \eqref{eq:kernel-I-2}, and \eqref{eq:par-I}, the proof will follow from the estimate
\begin{equation}
\label{eq:general-smooth}
\left\|\frac{\partial}{\partial t}(K_t(n+1)-K_t(n))\right\|_{\mathbb{B}_1}\le \frac{C}{n^2}, \qquad \text{for $n\not=0$.}
\end{equation}
Using \eqref{eq:partial-K}, we have
\begin{equation*}
\label{eq:partial-K-diff}
\frac{\partial}{\partial t}(K_t(n+1)-K_t(n))=\frac{1}{2}(K_t(n+2)-3K_t(n+1)+3K_t(n)-K_t(n-1)).
\end{equation*}
Integrating by parts three times in \eqref{eq:Schlafli} gives
\begin{align}
\label{eq:SchlafliIV}
I_{\nu}(z) &= -\frac{z^{\nu-3}}{\sqrt{\pi}\,2^{\nu-3}\Gamma(\nu-5/2)} \\*
\notag& \qquad \times \int_{-1}^1 e^{-zs} \,\frac{s(s^2z^2+3sz+3)}{z^2}(1-s^2)^{\nu-7/2}\,ds,
\quad \nu>\frac52.
\end{align}
Then, using \eqref{eq:Schlafli}, \eqref{eq:SchlafliII}, \eqref{eq:SchlafliIII}, and \eqref{eq:SchlafliIV} with $\nu=n-1$, $\nu=n$, $\nu=n+1$, and $\nu=n+2$, respectively, \eqref{eq:partial-K-diff} becomes
\begin{equation*}
\frac{\partial}{\partial t}(K_t(n+1)-K_t(n))=\frac{-1}{2}\left(3J_{1,t}(n)+3J_{2,t}(n)+J_{3,t}(n)\right),
\end{equation*}
where
\[
J_{1,t}(n)=\frac{t^{n-3}}{\sqrt{\pi}2^{n-1}\Gamma(n-1/2)}\int_{-1}^{1}e^{-t(1+s)}s(1-s^2)^{n-3/2}\, ds,
\]
\[
J_{2,t}(n)= \frac{t^{n-2}}{\sqrt{\pi}2^{n-1}\Gamma(n-1/2)}
\int_{-1}^{1}e^{-t(1+s)}s(1+s)(1-s^2)^{n-3/2}\, ds,
\]
and
\[
J_{3,t}(n)= \frac{t^{n-1}}{\sqrt{\pi}2^{n-1}\Gamma(n-1/2)}
\int_{-1}^{1}e^{-t(1+s)}(1+s)^3(1-s^2)^{n-3/2}\, ds.
\]
To estimate these inequalities we proceed as in the previous proposition. In fact, for $n\ge 4$,
\begin{multline*}
\|J_{1,t}(n)\|_{\mathbb{B}_1}^2
=\frac{4\Gamma(2n-4)}{\pi  (\Gamma(n-1/2))^2}\\\times\int_{0}^{1}\int_{0}^{1}\frac{(1-2u)(1-2v)(u(1-u))^{n-3/2}(v(1-v))^{n-3/2}}{(u+v)^{2n-4}}\, du \, dv,
\end{multline*}
and
\begin{multline*}
\|J_{1,t}(n)\|_{\mathbb{B}_1}^2
\le C\frac{\Gamma(2n-4)}{(\Gamma(n-1/2))^2}\left(\int_{0}^{1}v^2(1-v)^{n-3/2}\, dv\right)\\\times\left(\int_{0}^{\infty}\frac{z^{n-3/2}}{(1+z)^{2n-4}}\,dz\right)\le \frac{C}{n^6};
\end{multline*}
\begin{multline*}
\|J_{2,t}(n)\|_{\mathbb{B}_1}^2
=\frac{4\Gamma(2n-2)}{\pi  (\Gamma(n-1/2))^2}\\\times\int_{0}^{1}\int_{0}^{1}\frac{(2u-1)(2v-1)(uv)^{n-1/2}((1-u)(1-v))^{n-3/2}}{(u+v)^{2n-2}}\, du \, dv,
\end{multline*}
and
\begin{multline*}
\|J_{2,t}(n)\|_{\mathbb{B}_1}^2
\le C\frac{\Gamma(2n-2)}{(\Gamma(n-1/2))^2}\left(\int_{0}^{1}v^2(1-v)^{n-3/2}\, dv\right)\\\times\left(\int_{0}^{\infty}\frac{z^{n-1/2}}{(1+z)^{2n-2}}\,dz\right)\le \frac{C}{n^4};
\end{multline*}
and finally,
\begin{equation*}
\|J_{3,t}(n)\|_{\mathbb{B}_1}^2
=\frac{16\Gamma(2n)}{\pi  (\Gamma(n-1/2))^2}\int_{0}^{1}\int_{0}^{1}\frac{(uv)^{n+3/2}((1-u)(1-v))^{n-3/2}}{(u+v)^{2n}}\, du \, dv,
\end{equation*}
and
\begin{multline*}
\|J_{3,t}(n)\|_{\mathbb{B}_1}^2
\le C\frac{\Gamma(2n)}{(\Gamma(n-1/2))^2}\left(\int_{0}^{1}v^4(1-v)^{n-3/2}\, dv\right)\\\times\left(\int_{0}^{\infty}\frac{z^{n+3/2}}{(1+z)^{2n}}\,dz\right)\le \frac{C}{n^4}.
\end{multline*}
We deduce \eqref{eq:general-smooth} from the previous estimates for $n\ge 4$. The remainder cases can be proved as \eqref{eq:particular} in the previous proposition and then,
\[
\left\|\frac{\partial}{\partial t}(K_t(n+1)-K_t(n))\right\|_{\mathbb{B}_1}\le C, \qquad n=1,2,3.\qedhere
\]
\end{proof}

\section{Proof of Corollary \ref{cor:poisson}}
\label{sec:poisson}
First, it is easy to check that
\[
P^{(\alpha,\beta)}_tf(n)=\sum_{m=0}^{\infty}f(m)\mathcal{K}_t^{(\alpha,\beta)}(m,n),
\]
with
\[
\mathcal{K}_t^{(\alpha,\beta)}(m,n)=\int_{-1}^{1}e^{-t\sqrt{1-x}}p_m^{(\alpha,\beta)}(x)p_n^{(\alpha,\beta)}(x)\, d\mu_{\alpha,\beta}(x).
\]
Then, we have the following result for the $\mathfrak{g}_k^{(\alpha,\beta)}$-function which is the analogue of Lemma \ref{lem:gk-L2}.
\begin{lem}
  Let $\alpha,\beta\ge -1/2$ and $k\ge 1$. Then
\begin{equation}
\label{eq:gk-L2-Poi}
\|\mathfrak{g}_k^{(\alpha,\beta)}(f)\|^2_{\ell^2(\mathbb{N})}=\frac{\Gamma(2k)}{2^{2k}}\|f\|^2_{\ell^2(\mathbb{N})}.
\end{equation}
\end{lem}
This lemma can be proved following step by step the proof of Lemma \ref{lem:gk-L2}, so we omit the details. Now, using polarization, we deduce the identity
\[
\sum_{n=0}^{\infty}f(n)h(n)=\frac{2^{2k}}{\Gamma(2k)}\sum_{n=0}^{\infty} \int_{0}^{\infty}t^{2k-1}\left(\frac{\partial^k}{\partial t^k}P_t^{(\alpha,\beta)}f(n)\right)\left(\frac{\partial^k}{\partial t^k}P_t^{(\alpha,\beta)}h(n)\right)\, dt.
\]
From this fact, we obtain the inequality
\[
\|f\|_{\ell^p(\mathbb{N},w)}\le C \|\mathfrak{g}_k^{(\alpha,\beta)}(f)\|_{\ell^p(\mathbb{N},w)}
\]
from the direct inequality
\begin{equation}
\label{eq:gk-direct-Poi}
\|\mathfrak{g}_k^{(\alpha,\beta)}(f)\|_{\ell^p(\mathbb{N},w)}\le C \|f\|_{\ell^p(\mathbb{N},w)}
\end{equation}
as we did in the proof of Theorem \ref{th:main}. Finally, inequality \eqref{eq:gk-direct-Poi} is an immediate consequence of the following lemma.
\begin{lem}
\label{lem:gkPoi-gkHeat}
Let $\alpha,\beta>-1$, then
\[
\mathfrak{g}_{k}^{(\alpha,\beta)}(f)(n)\le \sum_{j=0}^{[k/2]} A_j g_{k-j}^{(\alpha,\beta)}(f)(n),
\]
where $A_j$ are some constants and $[\cdot]$ denotes the floor function.
\end{lem}
\begin{proof}
First, we observe that
\[
\frac{\partial^k}{\partial t^k}h\left(\frac{t^2}{4u}\right)=\sum_{j=0}^{[k/2]}B_j\left.\frac{\partial^{k-j} }{\partial s^{k-j}}h(s)\right|_{s=\frac{t^2}{4u}}\frac{t^{k-2j}}{(4u)^{k-j}},
\]
for some constants $B_j$. Then, from \eqref{eq:poisson}, we have
\[
\frac{\partial^k}{\partial t^k}P^{(\alpha,\beta)}_tf(n)=\frac{1}{\sqrt{\pi}}\sum_{j=0}^{[k/2]}B_j\int_{0}^{\infty}\frac{e^{-u}}{\sqrt{u}}
\left(\left.\frac{\partial^{k-j} }{\partial s^{k-j}}W_{s}^{(\alpha,\beta)}f(n)\right|_{s=\frac{t^2}{4u}}\right)\frac{t^{k-2j}}{(4u)^{k-j}}\, du
\]
and, by Minkowsk's integral inequality,
\[
\mathfrak{g}_k^{(\alpha,\beta)}(f)(n)\le \sum_{j=0}^{[k/2]}B_j P_j(n)
\]
where
\begin{multline*}
P_j(n)\\=\frac{1}{\sqrt{\pi}}\int_{0}^{\infty}\frac{e^{-u}}{\sqrt{u}(4u)^{k-j}}\left(\int_{0}^{\infty}t^{4k-4j-1}\left(\left.\frac{\partial^{k-j} }{\partial s^{k-j}}W_{s}^{(\alpha,\beta)}f(n)\right|_{s=\frac{t^2}{4u}}\right)^2\, dt\right)^{1/2}\, du.
\end{multline*}
Now, by using an appropriate change of variables, we have
\begin{align*}
P_j(n)&=\frac{1}{\sqrt{2\pi}}\int_{0}^{\infty}\frac{e^{-u}}{\sqrt{u}}\left(\int_{0}^{\infty}s^{2k-2j-1}\left(\frac{\partial^{k-j} }{\partial s^{k-j}}W_{s}^{(\alpha,\beta)}f(n)\right)^2\, ds\right)^{1/2}\, du\\&=\frac{1}{\sqrt{2}}g_{k-j}^{(\alpha,\beta)}(f)(n)
\end{align*}
and the result follows.
\end{proof}
\section{Proof of Theorem \ref{th:laplace-multi}}
\label{sec:Proof-Laplace}
We need only prove that
\begin{equation}
\label{eq:laplace}
g_1^{(\alpha,\beta)}(T_Mf)(n)\le C g_2^{(\alpha,\beta)}(f)(n),
\end{equation}
since by Theorem \ref{th:main} we get that
\[
\|T_Mf\|_{\ell^p(\mathbb{N},w)}\le C \|g_1^{(\alpha,\beta)}(T_Mf)\|_{\ell^p(\mathbb{N},w)}\le C \|g_2^{(\alpha,\beta)}(f)\|_{\ell^p(\mathbb{N},w)}\le C \|f\|_{\ell^p(\mathbb{N},w)}.
\]

Moreover, it is enough to prove \eqref{eq:laplace} for sequences in $c_{00}$, the space of sequences having a finite number of non-null terms. First, we have
\[
T_Mf(n)=-\int_{0}^{\infty}a(s)\frac{\partial}{\partial s}W_s^{(\alpha,\beta)}f(n)\, ds,
\]
which is an elementary consequence of the relation
\begin{multline*}
\int_{-1}^{1}M(1-x)p_m^{(\alpha,\beta)}(x)p_n^{(\alpha,\beta)}(x)\, d\mu_{\alpha,\beta}(x)\\
\begin{aligned}
&= \int_{0}^{\infty}a(s)\int_{-1}^{1}(1-x)e^{-s(1-x)}p_m^{(\alpha,\beta)}(x)p_n^{(\alpha,\beta)}(x)\, d\mu_{\alpha,\beta}(x)\, ds
\\&=-\int_{0}^{\infty}a(s)\frac{\partial}{\partial s}\int_{-1}^{1}e^{-s(1-x)}p_m^{(\alpha,\beta)}(x)p_n^{(\alpha,\beta)}(x)\, d\mu_{\alpha,\beta}(x)\, ds.
\end{aligned}
\end{multline*}
Then, applying the semigroup property of $W_t^{(\alpha,\beta)}$ we obtain
\begin{equation*}
W_t^{(\alpha,\beta)}(T_Mf)(n)=
-\int_{0}^{\infty}a(s)\frac{\partial}{\partial s}W_{s+t}^{(\alpha,\beta)}f(n)\, ds
\end{equation*}
and hence,
\begin{align*}
\frac{\partial}{\partial t}W_t^{(\alpha,\beta)}(T_Mf)(n)&=-\int_{0}^{\infty}a(s)\frac{\partial}{\partial t}\frac{\partial}{\partial s}W_{s+t}^{(\alpha,\beta)}f(n)\, ds\\&=-\int_{0}^{\infty}a(s)\frac{\partial^2}{\partial s^2}W_{s+t}^{(\alpha,\beta)}f(n)\, ds.
\end{align*}
In this way,
\begin{align*}
\left|\frac{\partial}{\partial t}W_t^{(\alpha,\beta)}(T_Mf)(n)\right|&\le C \int_{t}^{\infty}s\left|\frac{\partial^2}{\partial s^2}W_{s}^{(\alpha,\beta)}f(n)\right|\,\frac{ds}{s}\\&\le C t^{-1/2}\left(\int_{t}^{\infty}s^2\left|\frac{\partial^2}{\partial s^2}W_{s}^{(\alpha,\beta)}f(n)\right|^2\,ds\right)^{1/2}.
\end{align*}
Finally,
\begin{align*}
(g_1^{(\alpha,\beta)}(T_Mf)(n))^2&=\int_{0}^{\infty} t\left|\frac{\partial}{\partial t}W_t^{(\alpha,\beta)}(T_Mf)(n)\right|^2\, dt
\\& \le C
\int_{0}^{\infty}\int_{t}^{\infty}s^2\left|\frac{\partial^2}{\partial s^2}W_{s}^{(\alpha,\beta)}f(n)\right|^2\,ds\, dt\\
&=C
\int_{0}^{\infty}s^3\left|\frac{\partial^2}{\partial s^2}W_{s}^{(\alpha,\beta)}f(n)\right|^2\,ds=C (g_2^{(\alpha,\beta)}f(n))^2
\end{align*}
and the proof of \eqref{eq:laplace} is completed.

\end{document}